\documentclass[onecolumn,12pt]{elsarticle}

\usepackage{amssymb}

\newtheorem{theorem}{Theorem}[section]
\newtheorem{lemma}[theorem]{Lemma}
\newtheorem{remark}[theorem]{Remark}
\newtheorem{example}[theorem]{Example}
\newtheorem{corollary}[theorem]{Corollary}

\newenvironment{proof}[1][Proof:]{\begin{trivlist}
\item[\hskip \labelsep {\bfseries #1}]}{\end{trivlist}}

\journal{ISRN Algebra}

\begin{document}

\begin{frontmatter}

\title{A note on solutions of linear systems}

\medskip

\author{
Branko Male\v sevi\' c\mbox{\small ${\,}^{a}$},
Ivana Jovovi\' c\mbox{\small ${\,}^{a}$},   \\
Milica Makragi\' c\mbox{\small ${\,}^{a}$},
Biljana  Radi\v ci\' c\mbox{\small ${\,}^{b}$}
}

\medskip

\address[etf]{Faculty of Electrical Engineering, University of Belgrade,       \\
Bulevar kralja Aleksandra 73, 11000 Belgrade, Serbia}
\address[matf]{Faculty of Civil Engineering, University of  Belgrade,          \\
Bulevar kralja Aleksandra 73, 11000 Belgrade, Serbia}

\medskip

\begin{abstract}
In this paper we will consider Rohde's general form of $\{1\}$-inverse
of a matrix $A$.
The necessary and sufficient condition for consistency of a linear system $Ax=c$ will be represented.
We will also be concerned with the minimal number of free parameters in Penrose's formula
$x = A^{(1)} c + (I - A^{(1)}A)y$ for obtaining the general solution of the linear system.
This results will be applied for finding the general solution of various homogenous and non- -homogenous linear systems
as well as for different types of matrix equations.
\end{abstract}

\medskip

{\footnotesize
\begin{keyword}
\hspace{-0.5cm}
{\footnotesize
Generalized inverses, linear systems, matrix equations}
\end{keyword}}
\end{frontmatter}

\renewcommand{\thefootnote}{}

\footnotetext{
\hspace{- 8.80 mm}
{\it Email addresses:}

\smallskip

\hspace{-5.0 mm}
{\it  Branko Male\v sevi\' c} $<${\sl malesevic@etf.rs}$>$,
{\it  Ivana Jovovi\' c} $<${\sl ivana@etf.rs}$>$,
{\it  Milica Makragi\' c} $<${\sl milica.makragic@etf.rs}$>$,
{\it  Biljana  Radi\v ci\' c} $<${\sl biljana\textunderscore radicic@yahoo.com}$>$ }

\bigskip

\bigskip

\section{Introduction}

\medskip

In this paper we consider non-homogeneous linear system in $n$ variables
\begin{equation}
\label{Ax=c}
Ax=c,
\end{equation}
where $A$ is an $m \times n$ matrix over the field $\mathbb{C}$ of rank $a$ and $c$ is an $m \times 1$ matrix over $\mathbb{C}$. The set of all $m \times n$ matrices
over the complex field $\mathbb{C}$ will be denoted by $\mathbb{C}^{m \times n}$, $m, n \in \mathbb{N}$. The set of all $m \times n$ matrices over the complex field
$\mathbb{C}$ of rank $a$ will be denoted by $\mathbb{C}_{a}^{m \times n}$. For simplicity of notation, we will write $A_{i\rightarrow}$ ($A_{\downarrow j}$) for
the $i^{th}$ row (the $j^{th}$ column) of the matrix $A \in \mathbb{C}^{m \times n}$.

\break

Any matrix $X$ satisfying the equality $AXA=A$ is called $\{1\}$-inverse of $A$
and is denoted by $A^{(1)}$. The set of all $\{1\}$-inverses of the matrix $A$ is denoted by $A\{1\}$.
It can be shown that $A\{1\}$ is not empty.
If the $n \times n$ matrix $A$ is invertible,
then the equation $AXA=A$ has exactly one solution $A^{-1}$,
so the only $\{1\}$-inverse of the matrix $A$ is its inverse $A^{-1}$,
i.e. $A\{1\}$=$\{A^{-1}\}$. Otherwise, $\{1\}$-inverse of the matrix $A$ is not uniquely determined.
For more informations about $\{1\}$-inverses and various generalized inverses we recommend
A.Ben-Israel and T.N.E. Greville \cite{Ben--IsraelGreville03} and S.L. Campbell and C.D. Meyer \cite{CampbellMeyer09}.

For each matrix $A \in \mathbb{C}_{a}^{m \times n}$
there are regular matrices $P \in \mathbb{C}^{n \times n}$ and $Q \in \mathbb{C}^{m \times m}$ such that
\begin{equation}
\label{QAP=Ea}
QAP=E_{a}=\left[
\begin{array}{c|c}
I_{a} & 0    \\\hline
0     & 0
\end {array}
\right],
\end{equation}
where $I_{a}$ is $a \times a$ identity matrix.
It can be easily seen that every $\{1\}$-inverse of the matrix $A$ can be represented in the form
\begin{equation}
\label{A(1)}
A^{(1)} = P
\left[
\begin{array}{c|c}
I_{a} & U     \\\hline
V     & W
\end {array}
\right]Q
\end{equation}
where $U=[u_{ij}]$, $V=[v_{ij}]$ and $W=[w_{ij}]$ are arbitrary matrices of corresponding dimensions
$a \times (m-a)$, $(n-a) \times a$ and $(n-a) \times (m-a)$ with mutually independent entries,
see C. Rohde \cite{Rohde64} and V. Peri\' c \cite{Peric82}.

We will generalize the results of N. S. Urquhart \cite{Urquhart69}.
Firstly, we explore the minimal numbers of free parameters
in Penrose's formula
$$x = A^{(1)} c + (I - A^{(1)}A)y$$
for obtaining the general solution of the system (\ref{Ax=c}).
Then, we consider relations among the elements of $A^{(1)}$ to obtain
the general solution in the form $x=A^{(1)}c$ of the system (\ref{Ax=c}) for $c \neq 0$.
This construction has previously been used by B. Male\v sevi\' c and B. Radi\v ci\' c \cite{MalesevicRadicic11}
(see also \cite{MalesevicRadicic11b} and \cite{MalesevicRadicic12}).
At the end of this paper we will give an application of this results to the matrix equation $AXB=C$.

\section{The main result}

\medskip

In this section we indicate how technique of an $\{1\}$-inverse may be used to obtain the necessary
and sufficient condition for an existence of a general solution of a non-homogeneous linear system.

\begin{lemma}
\label{c'=Qc}
The non-homogeneous linear system {\rm (\ref{Ax=c})} has a solution if and only if
the last $m-a$ coordinates of the vector $c'=Qc$ are zeros, where
$Q \in \mathbb{C}^{m \times m}$ is regular matrix such that {\rm(\ref{QAP=Ea})} holds.
\end{lemma}

\begin{proof}
The proof follows immediately from Kroneker--Capelli theorem.
We provide a new proof of the lemma by using the $\{1\}$-inverse of the system matrix $A$.
The system (\ref{Ax=c}) has a solution if and only if $c=AA^{(1)}c$, see R.~Penrose \cite{Penrose55}.
Since $A^{(1)}$ is described by the equation (\ref{A(1)}), it follows that
$$
AA^{(1)} =
AP\left[
\begin{array}{c|c}
I_{a} & U     \\\hline
V     & W
\end{array}
\right]Q
 =
Q^{-1}\left[
\begin{array}{c|c}
I_{a} & U     \\\hline
0     & 0
\end{array}
\right]Q.
$$
Hence, we have the following equivalences
$$
\!\!
\begin{array}{lcl}
c=AA^{(1)}c
\!& \!\!\Longleftrightarrow\!\! &
(I \!-\! AA^{(1)})c \!=\! 0
\;\Longleftrightarrow\;
\!\left(Q^{-1}Q
-
Q^{-1}\!\left[
\begin{array}{c|c}
I_{a} & U     \\\hline
0     & 0
\end {array}
\right]\!Q\!\right)c \!=\! 0                                \\[3ex]
\!& \!\!\Longleftrightarrow\!\! &
Q^{-1}\left[
\begin{array}{c|c}
0     &  -U     \\\hline
0     &   I_{n-a}
\end {array}
\right]\underbrace{Qc}_{c'}=0
\Longleftrightarrow
\left[
\begin{array}{c|c}
0 &  -U     \\\hline
0 &   I_{n-a}
\end {array}
\right]c'=0          \\[2ex]
\!& \!\!\mathop{\Longleftrightarrow}\limits^{^{{\mbox {\tiny
  $\!\!c'=
  \left[
  \begin{array}{c}
  c_a'        \\
  c_{n-a}'
  \end{array}
  \right]$}}}}\!\! &
\left[
\begin{array}{c|c}
0     & -U     \\\hline
0     & I_{n-a}
\end{array}
\right]
\left[
\begin{array}{c}
c_a'        \\
c_{n-a}'
\end{array}
\right]=0
 \Longleftrightarrow  \left[
\begin{array}{r}
-Uc_{n-a}'        \\
 c_{n-a}'
\end{array}
\right]=0  \\[3ex]
\!& \!\!\Longleftrightarrow\!\! &
c_{n-a}'=0.
\end{array}
$$
Furthermore, we conclude
$c=AA^{(1)}c  \Longleftrightarrow c_{n-a}'=0.$
\qed
\end{proof}

\begin{theorem}
\label{main theorem 1}
The vector
$$x=A^{(1)}c+(I-A^{(1)}A)y,$$
$y \in \mathbb{C}^{n \times 1}$ is an arbitrary column,
is the general solution of the system {\rm (\ref{Ax=c})},
if and only if the $\{1\}$-inverse $A^{(1)}$
of the system matrix $A$ has the form {\rm(\ref{A(1)})} for arbitrary matrices $U$ and $W$
and the rows of the matrix $V(c'_a-y'_a)+y'_{(n-a)}$ are free parameters, where
$Qc=c'=
\left[
\begin{array}{c}
c'_{a}      \\\hline
0
\end {array}
\right]$
and
$P^{-1}y=y'=\left[
\begin{array}{c}
y'_{a}      \\\hline
y'_{n-a}
\end {array}
\right]$.
\end{theorem}

\begin{proof}
Since $\{1\}$-inverse $A^{(1)}$ of the matrix $A$ has the form (\ref{A(1)}),
the solution of the system $x=A^{(1)}c+(I-A^{(1)}A)y$ can be represented in the form
$$
\begin{array}{rcl}
x & = &
P\left[
\begin{array}{c|c}
I_{a} & U     \\\hline
V     & W
\end {array}
\right]Qc
+
\left(
I - P\left[
\begin{array}{c|c}
I_{a} & U     \\\hline
V     & W
\end {array}
\right]Q A
\right)y \\[2ex]
& = &
P\left[
\begin{array}{c|c}
I_{a} & U     \\\hline
V     & W
\end {array}
\right]c'
+
\left(
I - P\left[
\begin{array}{c|c}
I_{a} & U     \\\hline
V     & W
\end{array}
\right]Q A P P^{-1}
\right)y.
\end{array}
$$
According to Lemma \ref{c'=Qc} and from (\ref{QAP=Ea}) we have
$$
\begin{array}{rcl}
x   & = &
P\left[
\begin{array}{c|c}
I_{a} & U     \\\hline
V     & W
\end {array}
\right]
\left[
\begin{array}{c}
c'_{a}      \\\hline
0
\end{array}
\right]
+
\left(
I - P\left[
\begin{array}{c|c}
I_{a} & U     \\\hline
V     & W
\end{array}
\right]
\left[
\begin{array}{c|c}
I_{a} & 0     \\\hline
0     & 0
\end{array}
\right]
P^{-1}
\right)y.
\end{array}
$$
Furthermore, we obtain
$$
\begin{array}{rcl}
x  & = &
P\left[
\begin{array}{c}
 c'_{a}      \\\hline
Vc'_{a}
\end{array}
\right] +
\left(
I - P\left[
\begin{array}{c|c}
I_{a} & 0     \\\hline
V     & 0
\end{array}
\right]
P^{-1}
\right)
\left[
\begin{array}{c}
y_{a}      \\\hline
y_{n-a}
\end{array}
\right]     \\[2ex]
& = &
P\left[
\begin{array}{c}
 c'_{a}      \\\hline
Vc'_{a}
\end{array}
\right] +
\left(
PP^{-1} - P\left[
\begin{array}{c|c}
I_{a} & 0     \\\hline
V     & 0
\end{array}
\right]
P^{-1}
\right)
\left[
\begin{array}{c}
y_{a}      \\\hline
y_{n-a}
\end{array}
\right]   \\[2ex]
& = &
P\left[
\begin{array}{c}
 c'_{a}      \\\hline
Vc'_{a}
\end {array}
\right] +
P\left(
I - \left[
\begin{array}{c|c}
I_{a} & 0     \\\hline
V     & 0
\end {array}
\right]
\right)
P^{-1}
\left[
\begin{array}{c}
y_{a}      \\\hline
y_{n-a}
\end{array}
\right]    \\[2ex]
& = &
P\left[
\begin{array}{c}
 c'_{a}      \\\hline
Vc'_{a}
\end {array}
\right] +
P
\left[
\begin{array}{c|c}
 0     & 0     \\\hline
-V     & I_{n-a}
\end {array}
\right]
\left[
\begin{array}{c}
y'_{a}      \\\hline
y'_{n-a}
\end{array}
\right],
\end{array}
$$
where $y'=P^{-1}y$.
We now conclude
$$
\begin{array}{rclrcl}
x  & \!\!=\!\! &
P\left(\left[
\begin{array}{c}
 c'_{a}      \\\hline
Vc'_{a}
\end {array}
\right] +
\left[
\begin{array}{c}
0              \\\hline
-Vy'_{a}+y'_{n-a}
\end{array}
\right]
\right)
& \!\!=\!\! &
P\left[
\begin{array}{c}
 c'_{a}      \\\hline
V(c'_{a}-y'_{a})+y'_{n-a}
\end {array}
\right].
\end{array}
$$
Therefore, since matrix $P$ is regular we deduce that
$P\left[
\begin{array}{c}
 c'_{a}      \\\hline
V(c'_{a}-y'_{a})+y'_{n-a}
\end {array}
\right]$
is the general solution of the system (\ref{Ax=c}) if and only if the rows of the matrix $V(c'_{a}-y'_{a})+y'_{n-a}$
are $n-a$ free parameters.
\qed
\end{proof}

\begin{corollary}
\label{homogeneous}
The vector
$$x=(I-A^{(1)}A)y,$$
$y \in \mathbb{C}^{n \times 1}$ is an arbitrary column,
is the general solution of the homogeneous linear system $Ax=0$,
$A \in \mathbb{C}^{m \times n}$,
if and only if the $\{1\}$-inverse $A^{(1)}$
of the system matrix $A$ has the form {\rm(\ref{A(1)})} for arbitrary matrices $U$ and $W$
and the rows of the matrix $-Vy'_a+y'_{(n-a)}$ are free parameters, where
$P^{-1}y=y'=\left[
\begin{array}{c}
y'_{a}      \\\hline
y'_{n-a}
\end {array}
\right]$.
\end{corollary}



\begin{example}
Consider the homogeneous linear system
\begin{equation}
\label{example 1}
\begin{array}{rcl}
 x_1+2x_2+3x_3 &=& 0  \\
4x_1+5x_2+6x_3 &=& 0.
\end{array}
\end{equation}
The system matrix is
$$
A=
\left[
\begin{array}{ccc}
1 & 2 & 3 \\
4 & 5 & 6
\end{array}
\right].
$$
For regular matrices
$$
Q=
\left[
\begin{array}{rr}
 1  & 0  \\
-4  & 1
\end{array}
\right] \mbox{\,\,and\,\,}
P=
\left[
\begin{array}{rrr}
1 &  \frac{2}{3} &  1  \\[1ex]
0 & -\frac{1}{3} & -2  \\[1ex]
0 & 0            &  1
\end{array}
\right]
$$
equality (\ref{QAP=Ea}) holds.
Rohde's general $\{1\}$-inverse $A^{(1)}$ of the system matrix $A$ is of the form
$$
A^{(1)}=
P
\left[
\begin{array}{cc}
1      &  0      \\
0      &  1      \\
v_{11} & v_{12}
\end{array}
\right]
Q
$$
According to Corollary \ref{homogeneous} the general solution of the system (\ref{example 1}) is of the form
$$
x=
P\left[
\begin{array}{cc|c}
 0      &  0      & 0  \\
 0      &  0      & 0  \\\hline
-v_{11} & -v_{12} & 1
\end {array}
\right]P^{-1}
\left[
\begin{array}{c}
 y_1   \\
 y_2   \\
 y_3
\end {array}
\right],
$$
where
$$
P^{-1}=
\left[
\begin{array}{rrr}
1 &  2 &  3  \\
0 & -3 & -6  \\
0 & 0  &  1
\end{array}
\right].
$$
Therefore, we obtain
$$
\begin{array}{lcl}
x  & = &
P\left[
\begin{array}{cc|c}
 0      &  0      & 0  \\
 0      &  0      & 0  \\\hline
-v_{11} & -v_{12} & 1
\end {array}
\right]
\left[
\begin{array}{c}
  y_1+2y_2+3y_3   \\
-3y_2-6y_3        \\
  y_3
\end {array}
\right]                  \\[3ex]
   & = &
P
\left[
\begin{array}{c}
 0   \\
 0   \\
 -v_{11}y_1-(2v_{11}-3v_{12}y_2-(3v_{11}-6v_{12}-1)y_3
\end {array}
\right].
\end{array}
$$
If we take $\alpha=-v_{11}y_1-(2v_{11}-3v_{12}y_2-(3v_{11}-6v_{12}-1)y_3$ as a parameter
we get the general solution
$$
x=
\left[
\begin{array}{rrr}
1 &  \frac{2}{3} &  1  \\[1ex]
0 & -\frac{1}{3} & -2  \\[1ex]
0 & 0            &  1
\end{array}
\right]
\left[
\begin{array}{c}
 0   \\
 0   \\
 \alpha
\end {array}
\right]=
\left[
\begin{array}{c}
  \alpha   \\
-2\alpha   \\
  \alpha
\end {array}
\right].
$$
\end{example}

\begin{corollary}
\label{non-homogeneous}
The vector
$$
x=A^{(1)}c
$$
is the general solution of the system {\rm (\ref{Ax=c})},
if and only if the $\{1\}$-inverse $A^{(1)}$
of the system matrix $A$ has the form {\rm(\ref{A(1)})} for arbitrary matrices $U$ and $W$
and the rows of the matrix $Vc'_a$ are free parameters, where
$Qc=c'=
\left[
\begin{array}{c}
c'_{a}      \\\hline
0
\end {array}
\right]$.
\end{corollary}

\begin{remark}
Similar result can be found in paper B. Male\v sevi\' c and B. Radi\v ci\' c {\rm \cite{MalesevicRadicic11}}.
\end{remark}



\begin{example}
Consider the non-homogeneous linear system
\begin{equation}
\label{example 2}
\begin{array}{rcl}
 x_1+2x_2+3x_3 &=& 7  \\
4x_1+5x_2+6x_3 &=& 8.
\end{array}
\end{equation}
According to Corollary \ref{non-homogeneous} the general solution of the system (\ref{example 2}) is of the form
$$
x   =
P\left[
\begin{array}{cc}
 1      &  0       \\
 0      &  1       \\
 v_{11} &  v_{12}
\end {array}
\right]Q
\left[
\begin{array}{c}
 7   \\
 8
\end {array}
\right]
    =
P
\left[
\begin{array}{c}
  7   \\
 -20  \\
7v_{11}-20v_{12}
\end {array}
\right].
$$
If we take $\alpha=7v_{11}-20v_{12}$ as a parameter we obtain the general solution of the system
$$
x   =
P
\left[
\begin{array}{r}
  7   \\
 -20  \\
 \alpha
\end {array}
\right]
   =
\left[
\begin{array}{rrr}
1 &  \frac{2}{3} &  1  \\
0 & -\frac{1}{3} & -2  \\
0 & 0            &  1
\end{array}
\right]
\left[
\begin{array}{c}
  7   \\
 -20  \\
 \alpha
\end{array}
\right]
   =
\left[
\begin{array}{r}
-\frac{19}{3}+\alpha   \\
 \frac{20}{3}-2\alpha  \\
 \alpha
\end{array}
\right].
$$
\end{example}

We are now concerned with the matrix equation
\begin{equation}
\label{AX=C}
AX=C,
\end{equation}
where $A\in\mathbb{C}^{m\times n}$, $X\in\mathbb{C}^{n\times k}$ and $ C\in \mathbb{C}^{m\times k}$.

\begin{lemma}
\label{C'=QC}
The matrix equation {\rm (\ref{AX=C})} has a solution if and only if
the last $m-a$ rows of the matrix $C'=QC$ are zeros, where
$Q \in \mathbb{C}^{m \times m}$ is regular matrix such that {\rm(\ref{QAP=Ea})} holds.
\end{lemma}

\begin{proof}
If we write $X=[X_{\downarrow 1} \; X_{\downarrow 2} \; \ldots \; X_{\downarrow k}]$ and
$C=[C_{\downarrow 1} \; C_{\downarrow 2} \; \ldots \; C_{\downarrow k}]$,
then we can observe the matrix equation (\ref{AX=C}) as the system of matrix equations
$$
\begin{array}{c}
AX_{\downarrow 1}=C_{\downarrow 1}  \\[2.1 ex]
AX_{\downarrow 2}=C_{\downarrow 2}  \\[2.1 ex]
\vdots                              \\[2.1 ex]
AX_{\downarrow k}=C_{\downarrow k}.
\end{array}
$$
Each of the matrix equation $AX_{\downarrow i}=C_{\downarrow i}$, $1 \leq i \leq k$, by Lemma \ref{c'=Qc} has solution if and only if
the last $m-a$ coordinates of the vector $C'_{\downarrow i}=Q C_{\downarrow i}$ are zeros.
Thus, the previous system has solution if and only if the last $m-a$ rows of the matrix $C'=QC$ are zeros,
which establishes that the matrix equation (\ref{AX=C}) has solution if and only if all entries
of the last $m-a$ rows of the matrix $C'$ are zeros.
\qed
\end{proof}

\begin{theorem}
\label{theorem AX=C}
The matrix
$$X=A^{(1)}C+(I-A^{(1)}A)Y \in \mathbb{C}^{n \times k},$$
$Y \in \mathbb{C}^{n \times k}$ is an arbitrary matrix,
is the general solution of the matrix equation {\rm (\ref{AX=C})}
if and only if the $\{1\}$-inverse $A^{(1)}$
of the system matrix $A$ has the form {\rm(\ref{A(1)})} for arbitrary matrices $U$ and $W$
and the entries of the matrix
$$V(C'_a-Y'_a)+Y'_{(n-a)}$$
are mutually independent free parameters, where
$QC=C'=
\left[
\begin{array}{c}
C'_{a}      \\\hline
0
\end {array}
\right]$
and
$P^{-1}Y=Y'=\left[
\begin{array}{c}
Y'_{a}      \\\hline
Y'_{n-a}
\end {array}
\right]$.
\end{theorem}

\begin{proof}
Applying the Theorem \ref{main theorem 1} on the each system $AX_{\downarrow i}=C_{\downarrow i}$, $1 \leq i \leq k$,
we obtain that
$$X_{\downarrow i}=P\left[
\begin{array}{c}
 C'_{a \downarrow i}      \\\hline
V(C'_{a \downarrow i}-Y'_{a \downarrow i})+Y'_{n-a \downarrow i}
\end{array}
\right]$$
is the general solution of the system if and only if the rows of the matrix
$V(C'_{a \downarrow i}-Y'_{a \downarrow i})+Y'_{n-a \downarrow i}$
are $n-a$ free parameters.
Assembling these individual solutions together we get that
$$
X=P\left[
\begin{array}{c}
 C'_{a}            \\\hline
V(C'_{a}-Y'_{a})+Y'_{n-a}
\end{array}
\right]
$$
is the general solution of the matrix equation $(\ref{AX=C})$
if and only if entries of the matrix $V(C'_{a}-Y'_{a})+Y'_{n-a}$
are $(n-a)k$ mutually independent free parameters.
\qed
\end{proof}

From now on we proceed with the study of the non-homogeneous linear system of the form
\begin{equation}
\label{xB=d}
xB=d,
\end{equation}
where $B$ is an $n \times m$ matrix over the field $\mathbb{C}$ of rank $b$ and $d$ is an $1 \times m$ matrix
over $\mathbb{C}$. Let $R \in \mathbb{C}^{n \times n}$  and $S \in \mathbb{C}^{m \times m}$ be regular matrices
such that
\begin{equation}
\label{RBS=Eb}
RBS=E_{b}=\left[
\begin{array}{c|c}
I_{b} & 0    \\\hline
0     & 0
\end {array}
\right].
\end{equation}
An $\{1\}$-inverse of the matrix $B$ can be represented in the Rohde's form
\begin{equation}
\label{B(1)}
B^{(1)} =
S\left[
\begin{array}{c|c}
I_{b} & M     \\\hline
N     & K
\end {array}
\right]R
\end{equation}
where $M=[m_{ij}]$, $N=[n_{ij}]$ and $K=[k_{ij}]$
are arbitrary matrices of corresponding dimensions $b \times (n-b)$, $(m-b) \times b$ and $(m-b) \times (n-b)$
with mutually independent entries.

\begin{lemma}
\label{d'=dS}
The non-homogeneous linear system {\rm (\ref{xB=d})} has a solution if and only if
the last $m-b$ elements of the row $d'=dS$ are zeros, where
$S \in \mathbb{C}^{m \times m}$ is regular matrix such that {\rm(\ref{RBS=Eb})} holds.
\end{lemma}

\begin{proof}
By transposing the system (\ref{xB=d}) we obtain system $B^{T}x^{T}=d^{T}$ and
by transposing the matrix equation (\ref{RBS=Eb}) we obtain that $S^{T}B^{T}R^{T}=E_{b}$.
According to Lemma \ref{c'=Qc} the system $B^{T}x^{T}=d^{T}$ has solution if and only if
the last $m-b$ coordinates of the vector $S^{T}d^{T}$ are zeros, i.e. if and only if
the last $m-b$ elements of the row $d'=dS$ are zeros.
\qed
\end{proof}

\begin{theorem}
\label{main theorem 2}
The row
$$x=dB^{(1)}+y(I-BB^{(1)}),$$
$y \in \mathbb{C}^{1 \times n}$ is an arbitrary row,
is the general solution of the system {\rm (\ref{xB=d})},
if and only if the $\{1\}$-inverse $B^{(1)}$
of the system matrix $B$ has the form {\rm(\ref{B(1)})} for arbitrary matrices $N$ and $K$
and the columns of the matrix $(d'_b-y'_b)M+y'_{n-b}$ are free parameters, where
$dS=d'=
\left[ d'_{b} \;|\; 0 \right]$
and
$yR^{-1}=y'=\left[ y'_{b} \;|\; y'_{n-b} \right]$.
\end{theorem}

\begin{proof}
The basic idea of the proof is to transpose the system (\ref{xB=d}) and
to apply the Theorem \ref{main theorem 1}. The $\{1\}$-inverse of the matrix
$B^{T}$ is equal to a transpose of the $\{1\}$-inverse of the matrix $B$.
Hence, we have
$$(B^{T})^{(1)}=(B^{(1)})^{T}=
\left(S\left[
\begin{array}{c|c}
I_{b} & M     \\\hline
N     & K
\end {array}
\right]R\right)^{T}=
R^{T}
\left[
\begin{array}{c|c}
I_{b} & N^{T}     \\\hline  \\[-2ex]
M^{T} & K^{T}
\end {array}
\right]S^{T}.
$$
We can now proceed analogously to the proof of the Theorem \ref{main theorem 1} to obtain that
$$x^{T}=R^{T}\left[
\begin{array}{c}
 d'^{\,T}_{b}      \\\hline \\[-2ex]
M^{T}(d'^{\,T}_{b}-y'^{\,T}_{b})+y'^{\,T}_{n-b}
\end {array}
\right]$$
is the general solution of the system $B^{T}x^{T}=d^{T}$ if and only if
the rows of the matrix $M^{T}(d'^{\,T}_{b}-y'^{\,T}_{b})+y'^{\,T}_{n-b}$
are $n-b$ free parameters.
Therefore,
$$x=
\left[
d'_{b}   \;|\;
(d'_{b}-y'_{b})M+y'_{n-b}
\right]R
$$
is the general solution of the system (\ref{xB=d}) if and only if the columns of the matrix $(d'_{b}-y'_{b})M+y'_{n-b}$
are $n-b$ free parameters.
\qed
\end{proof}

\noindent
Analogous corollaries hold for the Theorem \ref{main theorem 2}.

\bigskip

We now deal with the matrix equation
\begin{equation}
\label{XB=D}
XB=D,
\end{equation}
where $X\in\mathbb{C}^{k \times n}$, $B\in\mathbb{C}^{n \times m}$  and $D\in\mathbb{C}^{k \times m}$.

\begin{lemma}
\label{D'=DS}
The matrix equation {\rm (\ref{XB=D})} has a solution if and only if
the last $m-b$ columns of the matrix $D'=DS$ are zeros, where
$S \in \mathbb{C}^{m \times m}$ is regular matrix such that {\rm(\ref{RBS=Eb})} holds.
\end{lemma}

\begin{theorem}
\label{theorem XB=D}
The matrix
$$X=DB^{(1)}+Y(I-BB^{(1)}) \in \mathbb{C}^{k \times n},$$
$Y \in \mathbb{C}^{k \times n}$ is an arbitrary matrix,
is the general solution of the matrix equation {\rm (\ref{XB=D})}
if and only if the $\{1\}$-inverse $B^{(1)}$
of the system matrix $B$ has the form {\rm(\ref{B(1)})} for arbitrary matrices $N$ and $K$
and the entries of the matrix
$$(D'_b-Y'_b)M+Y'_{(n-b)}$$
are mutually independent free parameters, where
$DS=D'=
\left[ D'_{b} \;|\; 0 \right]$
and
$YR^{-1}=Y'=\left[ Y'_{b} \;|\; Y'_{n-b} \right]$.
\end{theorem}

\section{An application}

\medskip

In this section we will briefly sketch properties of the general solution of the matrix equation
\begin{equation}
\label{AXB=C}
AXB=C,
\end{equation}
where $A\in \mathbb{C}^{m \times n}$, $X \in \mathbb{C}^{n \times k}$,
$B \in \mathbb{C}^{k \times l}$ and $C \in \mathbb{C}^{m \times l}$.
If we denote by $Y$ matrix product $XB$, then the matrix equation (\ref{AXB=C}) becomes
\begin{equation}
\label{AY=C}
AY=C.
\end{equation}
According to the Theorem \ref{theorem AX=C} the general solution of the system (\ref{AY=C})
can be presented as a product of the matrix $P$ and the matrix which has the first $a=rank(A)$ rows
same as the matrix $QC$ and the elements of the last $m-a$ rows are $(m-a)n$ mutually independent free parameters,
$P$ and $Q$ are regular matrices such that $QAP=E_{a}$.
Thus, we are now turning on to the system of the form
\begin{equation}
\label{XB=D new}
XB=D.
\end{equation}
By the Theorem \ref{theorem XB=D} we conclude that the general solution of the system (\ref{XB=D new})
can be presented as a product of the matrix which has the first $b=rank(B)$ columns equal
to the first $b$ columns of the matrix $DS$ and the rest of the columns have
mutually independent free parameters as entries, and the matrix $R$,
for regular matrices $R$ and $S$ such that $RBS=E_{b}$.
Therefore, the general solution of the system (\ref{AXB=C}) is of the form
$$
X = P
\left[
\begin{array}{c|c}
G_{ab} & F     \\\hline
H      & L
\end {array}
\right] R,
$$
where $G_{ab}$ is a submatrix of the matrix $QCS$ corresponding to the first $a$ rows and the first $b$ columns
and the entries of the matrices $F$, $H$ and $L$ are $nk-ab$ free parameters.
We will illustrate this on the following example.

\begin{example}
We consider the matrix equation
$$AXB=C,$$
where
$A  = \left[
\begin{array}{rr}
 1 & -2   \\
-2 &  4
\end{array}
\right]$,
$B=\left[
\begin{array}{rrr}
1 & 2 & 1 \\
1 & 2 & 1 \\
1 & 2 & 1
\end{array}
\right]$
and
$C=\left[
\begin{array}{rrr}
 1  &  2 &  1 \\
-2  & -4 & -2
\end {array}
\right]$.
If we take $Y=XB$, we obtain the system
$$AY=C.$$
It is easy to check that the matrix $A$ is of the rank $a=1$ and for matrices
$Q = \left[
\begin{array}{cc}
1 & 0   \\
2 & 1
\end{array}
\right]$
and
$P=\left[
\begin{array}{cc}
1 & 2  \\
0 & 1
\end {array}
\right]$
the equality $QAP=E_a$ holds.
Based on the Theorem \ref{theorem AX=C}, the equation $AY=C$ can be rewritten in the system form
$$
\begin{array}{rcl}
AY_{\downarrow1} & = &
\left[
\begin{array}{r}
 1  \\
-2
\end{array}
\right]                    \\[2ex]
AY_{\downarrow2} & = &
\left[
\begin{array}{r}
 2  \\
-4
\end{array}
\right]                     \\[2ex]
AY_{\downarrow3} & = &
\left[
\begin{array}{r}
 1  \\
-2
\end{array}
\right].
\end{array}
$$
Combining the Theorem \ref{main theorem 1} with the equality
$$\left[
\begin{array}{ccc}
c'_{1} & c'_{2} & c'_{3} \\
0      & 0      & 0
\end{array}
\right]=
\left[
\begin{array}{cc}
1 & 0     \\
2 &  1
\end{array}
\right]
\left[
\begin{array}{rrr}
1  &  2 & 1   \\
-2 & -4 & -2
\end{array}
\right]=
\left[
\begin{array}{rrr}
1  & 2 & 1   \\
0  & 0 & 0
\end{array}
\right]$$
yields
$$\begin{array}{rcl}
Y_{\downarrow1} & = &
P\left[
\begin{array}{c}
1                                    \\
\underbrace{v-vz_{11}+z_{21}}_{\alpha}
\end{array}
\right]                             \\[3ex]
Y_{\downarrow2} & = &
P\left[
\begin{array}{c}
2                                     \\
\underbrace{2v-2vz_{12}+z_{22}}_{\beta}
\end{array}
\right]                             \\[3ex]
Y_{\downarrow3} & = &
P\left[
\begin{array}{c}
1                                    \\
\underbrace{v-vz_{13}+z_{23}}_{\gamma}
\end{array}
\right],
\end{array}
$$
for an arbitrary matrix
$Z=\left[
\begin{array}{rrr}
z_{11} &  z_{12} & z_{13}   \\
z_{21} & -z_{22} & z_{23}
\end{array}
\right].$
Therefore, the general solution of the system $AY=C$ is
$$
Y  =
P\left[
\begin{array}{ccc}
1      &  2    & 1 \\
\alpha & \beta & \gamma
\end{array}
\right].
$$
From now on, we consider the system
$$XB=D$$
for
$$D =
P\left[
\begin{array}{ccc}
1      &  2    & 1 \\
\alpha & \beta & \gamma
\end{array}
\right] =
\left[
\begin{array}{ccc}
1+ 2\alpha & 2+2\beta &  1+2\gamma \\
    \alpha &    \beta &     \gamma
\end{array}
\right].
$$
There are regular matrices
$R  =
\left[
\begin{array}{rrr}
 1 & 0 & 0 \\
-1 & 1 & 0 \\
-1 & 0 & 1
\end{array}
\right]$
and
$S  =
\left[
\begin{array}{rrr}
1 & -2 & -1 \\
0 &  1 &  0  \\
0 &  0 &  1
\end{array}
\right]$
such that $RBS=E_b$ holds.
Since the rank of the matrix $B$ is $b=1$, according to the Lemma \ref{D'=DS}
all entries of the last two columns of the matrix $D'=DS$ are zeros,
i.e. we have $\gamma=\alpha$, $\beta=2\alpha$.
Hence, we get that the matrix $D'$ is of the form
$D' =
\left[
\begin{array}{ccc}
1+ 2\alpha & 0 &  0 \\
    \alpha & 0 &  0
\end{array}
\right]$.
Applying the Theorem \ref{theorem XB=D}, we obtain
$$X  =
\left[
\begin{array}{ccc}
1+ 2\alpha & \underbrace{(1+2\alpha-t_{11})m_{11}+t_{12}}_{\beta{1}}   & \underbrace{(1+2\alpha-t_{11})m_{12}+t_{13}}_{\beta{2}} \\
    \alpha & \underbrace{(   \alpha-t_{21})m_{11}+t_{22}}_{\gamma_{1}} & \underbrace{(   \alpha-t_{12})m_{12}+t_{23}}_{\gamma_{2}}
\end{array}
\right]R,
$$
for an arbitrary matrix
$T  =
\left[
\begin{array}{ccc}
t_{11} & t_{12} & t_{13} \\
t_{21} & t_{22} & t_{23}
\end{array}
\right]$.
Finally, the solution of the system $AXB=C$ is
$$X  =
\left[
\begin{array}{ccc}
1+ 2\alpha-\beta_{1}-\beta_{2}   & \beta_{1} &  \beta_{2}  \\
    \alpha-\gamma_{1}-\gamma_{2} & \gamma_{1} &  \gamma_{2}
\end{array}
\right].
$$
\end{example}

\bigskip

\bigskip

\noindent
\mbox{\textsc{Acknowledgment.}} Research is partially supported by the Ministry of Science and Education of the Republic of Serbia,
Grant No.174032.

\end{document}